\documentclass[reqno,12pt]{amsart}
\usepackage{amsmath, latexsym, amsfonts, amssymb, amsthm, amscd}
\usepackage{multirow}
\usepackage{hyperref}
\usepackage{mathrsfs}
\usepackage[T1]{fontenc}
\usepackage{comment}
\usepackage{xcolor}

 \usepackage{graphicx}
\usepackage[utf8]{inputenc} 

\usepackage{datetime}

\usepackage{cancel}
\usepackage{graphics,epsf,psfrag}
\numberwithin{equation}{section}

\newtheorem{theorem}{Theorem}[section]
\newtheorem{lemma}[theorem]{Lemma}
\newtheorem{example}[theorem]{Example}

\newtheorem{cor}[theorem]{Corollary}
\newtheorem{rem}[theorem]{Remark}
\newtheorem{definition}[theorem]{Definition}

\newcommand{\R}{\mathbb{R}}

\DeclareMathSymbol{\leqslant}{\mathalpha}{AMSa}{"36} % nicer `smaller or equal'
\DeclareMathSymbol{\geqslant}{\mathalpha}{AMSa}{"3E} % nicer `larger or equal'
\DeclareMathSymbol{\eset}{\mathalpha}{AMSb}{"3F}     % nicer `emptyset'
\renewcommand{\leq}{\;\leqslant\;}                   % redef. of < or =
\renewcommand{\geq}{\;\geqslant\;}                   % redef. of > or =
\newcommand{\dd}{\,\text{\rm d}}             % a straight d for differentials
\newcommand{\Div}{\mathrm{div}}

\newcommand{\norm}[1]{\left\lVert#1\right\rVert}
\newcommand{\abs}[1]{\left\lvert#1\right\rvert}
\newcommand{\seminorm}[1]{\left[#1\right]}

\title[Regularization by noise for the evolutionary $p$-Laplace equation]{A pathwise Regularization by noise phenomenon for the evolutionary $p$-Laplace equation}

\author{Florian Bechtold, J\" orn Wichmann}
\address{Fakultät für Mathematik, Universität Bielefeld, 33501 Bielefeld, Germany}
\email{fbechtold@math.uni-bielefeld.de, jwichmann@math.uni-bielefeld.de}
\begin{document}
\begin{abstract}
    We study an evolutionary $p$-Laplace problem whose potential is subject to a translation in time. Provided the trajectory along which the potential is translated admits a sufficiently regular local time, we establish existence of solutions to the problem for singular potentials for which a priori bounds in  classical approaches break down, thereby establishing a pathwise regularization by noise phenomena for this non-linear problem.
\end{abstract}
\maketitle
%\tableofcontents
\section{Introduction}
Let $\Omega$ be a bounded Lipschitz domain in $\R^d$, $T>0$ and $N \in \mathbb{N}$. We are interested in the evolutionary $p$-Laplace system
\begin{alignat}{2} \label{Original problem}
    \begin{aligned}
\partial_t u - \Div \,S(\nabla u) &= b(u) \quad &&\text{ on } \Omega \times [0,T], \\
u|_\Omega &=0 \quad &&\text{ on } \partial\Omega\times [0,T],\\
        u(0, \cdot)&=u_0 \quad &&\text{ on } \Omega,
    \end{aligned}
\end{alignat}
where $S(\xi) = \abs{\xi}^{p-2} \xi \in \R^{d\times N}$, $p\in (1,\infty)$ and $b:\mathbb{R}^N \to \mathbb{R}^N$. Equation~\eqref{Original problem} naturally arises as a perturbed gradient flow of the energy
\begin{align*}
    \mathcal{J}(u) =\frac{1}{p}  \int_{\Omega}\abs{\nabla u}^p \dd x
\end{align*}
on the Sobolev space $W^{1,p}_{0}(\Omega)$ of weakly differentiable vectorfields with vanishing trace. Questions related to the well-posedness of~\eqref{Original problem} are highly linked to regularity of $b \circ u$ and therefore $b$ as well as the notion of solutions. In particular, in the absence of certain growth conditions, the study of \eqref{Original problem} might be obstructed. Let us elaborate on this point in the following subsection.
\subsection*{Monotone operator theory}
The $p$-Laplace operator 
\begin{align*}
    Au:= -\Div\, S(\nabla u)
\end{align*}
 is a prominent example of a maximal monotone operator. The famous theory of monotone operators traces back to the early works of Minty~\cite{MR169064} and Browder~\cite{MR156204}. It inspired many mathematicians to study well-posedness of monotone evolution equations and perturbations of it, see e.g.~\cite{MR216342,MR226230,MR0259693,
 MR477358,MR675911,CH16,MR4041276}. A nice overview on monotone operators can be found e.g. in the books of Brezis~\cite{MR0348562}, Barbu~\cite{MR0390843} and Zeidler~\cite{MR1033498}.

 Let us briefly recall some main ideas and in particular function spaces that naturally appear in the study of \eqref{Original problem}. Towards this end, let us set $b \circ u = f$ and forget briefly that $f$ depends on $u$ and proceed formally. We can define the evolutionary problem as a steady problem
 \begin{align} \label{eq:General}
     B u := \partial_t u - \Div S(\nabla u) =  f.
 \end{align}
In this setting, we now need to find a suitable function space such that $B$ is surjective. First, note that for sufficiently smooth $u$ it holds
\begin{align*}
\langle Bu, u \rangle = \norm{u(T)}_{L^2}^2 - \norm{u(0)}_{L^2}^2 + \int_0^T \norm{\nabla u(s)}_{L^p}^p \dd s.
\end{align*}
Therefore, a natural ansatz space to consider is
\begin{align*}
    \mathbb{X} := C\big([0,T];L^2(\Omega)\big) \cap L^p\big(0,T;W^{1,p}_0(\Omega)\big).
\end{align*}
Indeed, monotone operator theory implies that we can solve~\eqref{eq:General} in the weak sense for $f \in \mathbb{X}'$ with a quantified energy estimate
\begin{align*}
    \norm{u}_{\mathbb{X}} \lesssim \norm{f}_{\mathbb{X}'}.
\end{align*}
Additionally, $u$ enjoys some time regularity as
\begin{align*}
    \norm{\partial_t u}_{\mathbb{X}'} \leq \norm{\nabla u}_{L^p(0;T;L^p(\Omega))}^{p-1} + \norm{f}_{\mathbb{X}'}.
\end{align*}
Overall, we find that the solution map
\begin{align*}
    L^2(\Omega) \times \mathbb{X}' \ni (u_0,f) \mapsto u \in \mathbb{Y},
\end{align*}
where 
\begin{align*}
   \mathbb{Y}  = \left\{ v \in \mathbb{X} \big| \, v(0) = u_0 \, , \, \partial_t v \in \mathbb{X}' \right\} 
\end{align*}
is bounded, i.e., 
\begin{equation} \label{intro:weakEnergy}
        \norm{u}_{\mathbb{X}} + \norm{\partial_t u}_{\mathbb{X}'} \lesssim \norm{f}_{\mathbb{X}'}+ \norm{u_0}_{L^2(\Omega)}\footnote{We neglect additional difficulties due to the miss match of the exponents}.
\end{equation}

A similar approach can be employed to construct weak solutions to~\eqref{Original problem} if $f = b\circ u$, provided that we can close the energy estimate~\eqref{intro:weakEnergy}. That is we need to remove the unknown solution $u$ from the right hand side. This can be done by imposing a growth condition on $b$, e.g., for all $u \in \mathbb{R}^N$
\begin{align}
    \abs{b(u) \cdot u} \lesssim \abs{u}^r + 1,
    \label{growth intro}
\end{align}
where $r \in \mathbb{R}$. Conceptually the cases~$r \geq 1$ and~$r < 1$ are completely different. The former case bounds the growth of~$b$ for large values of~$\abs{u}$. The latter one treats potentials whose singularities are located in the origin. The classical approach can cover the range~$r \in (0, p\vee 2]$.

If the growth at infinity is to strong in comparison to the dissipation of the $p$-Laplace operator solutions may blow up in finite time. This was already observed in the linear scalar case, $p = 2$ and $N=1$, by Fujita~\cite{MR0269995}. It was later extended by Tsutsumi in~\cite{MR0312079} to the non-linear scalar case~$S(\nabla u)_i = \abs{\partial_i u}^{p-2} \partial_i u$ for the prototype perturbation $b(u) = \max( u^{r-1},0)$. This leads to an upper bound $r \leq p \vee 2$.

The lower bound on the growth rate can be seen by choosing $b(u) =-|u|^{\eta - 1}u$ for some $ r= \eta + 1 \leq 0$, cf. Example~\ref{ex:Main}. In other words only minor singularities in zero can be treated with the classical weak solution approach.
\\
\\
In the following, we intend to study the effect of translations of the potential $b$ along so called "regularizing paths" $w$, that is we are interested in the problem
\begin{alignat}{2}\label{regularized problem}
    \begin{aligned}
    \partial_t u - \Div S(\nabla u) &= b(u-w) \quad &&\text{ on } \Omega \times [0,T], \\
u|_\Omega &=0 \quad &&\text{ on } \partial\Omega\times [0,T],\\
        u(0, \cdot)&=u_0 \quad &&\text{ on } \Omega.
    \end{aligned}
\end{alignat}

From a physical point of view, translating the potential in time can be interpreted as uncertainty as to the location of the origin of the potential. By a "regularizing path", we will understand a continuous path $w$ that admits a sufficiently regular local time\footnote{We refer to the corresponding section in the Appendix for the definition of the local time of a path and related concepts.} such that techniques from pathwise regularization by noise
 \` a la \cite{gubicat}, \cite{galeati2020noiseless}, \cite{harang2020cinfinity} become applicable. Let us briefly sketch some main ideas of these approaches, in particular in the spirit of \cite{harang2020cinfinity}, which we will be able to employ also in the study of \eqref{regularized problem}. 
 
 \subsection*{Pathwise regularization by noise in a nutshell}
The starting point for investigations of this type usually consists in the study of averaging operators along a path $w$, given by
 \[
(T^{-w}_tb)(u):= \int_0^tb(u-w_s)ds.
 \]
As was recognized already in  \cite{gubicat}, paths enjoying so called $\rho$ irregularity \cite[Definition 1.3]{gubicat}, i.e. paths whose Fourier transform of the occupation measure decrease sufficiently rapidly lead to a regularization effect in that function $u\mapsto (T^{-w}_tb)(u)$ will enjoy higher regularity than $b$. The reason for this gain of regularity can be made clear as follows: Assuming $w$ to also admit a local time $L$, we can rewrite the averaging operator thanks to the occupation times formula as 
\[
\int_0^tb(u-w_s)ds=\int_{\R^N}b(u-z)L_t(z)dz=(b*L_t)(u).
\]
If the Fourier transform of the occupation measure is decreasing rapidly in a certain quantifiable sense, the local time $L$ will enjoy some high quantifiable spatial regularity. By Young's inequality (see for example \eqref{young} or \cite{kuhn2021convolution} for a generalization to the Besov space setting), this implies that the regularity of $T^{-w}_tb$ will essentially increase by the regularity of $L_t$ with respect to the regularity of $b$. While the canonical paths for which this gain in regularity can be quantified are realizations of fractional Brownian motion \cite{gubicat}, \cite{galeati2020noiseless}, \cite{Gerencs_r_2022}, $\rho$ irregularity is in fact a typical property among H\"older continuous paths in the sense of prevalence \cite{prevalence}.\\
\\
Taking this observation of increased regularity of the averaging operator as a starting point, one can exploit this local gain of regularity to further study Riemann-sum type expressions of the form
\[
\mathcal{I}^n_{T}=\sum_{[s,t]\in \mathcal{P}^n([0,T])}(T^{-w}_{s,t}b)(u_s)
\]
for partitions $\mathcal{P}^n([0,T])$ of $[0,T]$ and a continuous function $u$. The tool that then ensures the convergence of such Riemann-sums in the limit $|\mathcal{P}^n|\to 0$ is Gubinelli's Sewing Lemma \cite{gubi}, \cite{frizhairer}, also cited in the Appendix as Lemma \ref{sewing}. Let us stress at this point already that since the Sewing Lemma is formulated in a "H\" older-space setting", convergence of $(\mathcal{I}^n)_n$ will however always require at least some H\"older regularity of the function $u$. Provided this is available, i.e. $u$ is sufficiently H\"older regular and the averaging operator $(T^{-w}b)$ sufficiently regular in time and space, the Sewing Lemma ensures the convergence of $(\mathcal{I}^n)_n$ as $|\mathcal{P}^n|\to 0$. Note that as this convergence eventually only requires information on the regularity of the averaging operator and not the regularity of $b$, this construction (which can be shown to coincide with the classical Lebesgue integral for regular $b$) naturally extends the definition of the Lebesgue integral to irregular and potentially even distributional $b$, i.e. one can define
\[
\mathcal{I}_{T}=\lim_{|\mathcal{P}^n|\to 0}\sum_{[s,t]\in \mathcal{P}^n([0,T])}(T^{-w}_{s,t}b)(u_s)=:\int_0^Tb(u_s-w_s)ds
\]
Moreover, even in the case of a regular nonlinearity $b$ where the above can be defined as a classical Lebesgue integral, this alternative definition provides alternative a priori bounds through the Sewing Lemma of which we will crucially make use in our work. 
\subsection*{Application to the p-Laplace equation with shifted potential}
From the above considerations and on the background of classical monotone operator theory, a first approach in studying \eqref{regularized problem} might consist in investigating the weak formulation
\begin{equation*}
    \langle u_t-u_s, \varphi\rangle+\int_s^t \langle S(\nabla u_r), \nabla \varphi\rangle dr=\int_s^t\langle b(u_r-w_r), \varphi \rangle dr,
\end{equation*}
where the right hand side should be interpreted as 
\[
\int_s^t\langle b(u_r-w_r), \varphi \rangle dr:=\lim_{|\mathcal{P}^n|\to 0}\sum_{[s',t']\in \mathcal{P}^n([s,t])}\langle (T^{-w}_{s',t'}b)(u_{s'}), \varphi\rangle.
\]
Note however that the classical monotone operator approach to this problem, working on the Gelfand triple $\big(W^{1,p}_0(\Omega)\cap L^2(\Omega), L^2(\Omega), (W^{1,p}_0(\Omega)\cap L^2(\Omega))'\big)$, only yields a priori bounds that permit to conclude $u\in C([0,T], L^2(\Omega))$. In particular, no additional H\"older regularity in time on this spatial regularity scale is obtained, meaning the Sewing argument can't be closed making  the right hand side ill defined. To circumvent this problem, we employ a strong formulation to the problem, i.e. we strive for solutions $u$ that satisfy
\begin{equation}
    \label{strong formulation}
    u_t-u_s-\int_s^t \mbox{div} S(\nabla u_r)dr=\int_s^t b(u_r-w_r)dr
\end{equation}
understood as an equality in $L^2(\Omega)$ and where for singular $b$ the right hand side is understood in the sense 
\begin{equation}
    \int_s^t b(u_r-w_r):=\lim_{|\mathcal{P}^n|\to 0}\sum_{[s',t']\in \mathcal{P}^n([s,t])} (T^{-w}_{s',t'}b)(u_{s'}),
    \label{robustification}
\end{equation}
where the convergence on the right hand side holds in $L^2(\Omega)$, uniformly in time on $[0,T]$. Note that while such strong solutions naturally require higher regularity of the initial condition, namely $u_0\in W^{1,p}_0(\Omega)$, they allow for further a priori bounds in $C^{0,1/2}\big([0,T];L^2(\Omega)\big)$ provided $b$ is sufficiently regular (refer to section \ref{strong solutions}). Having for regular $b$ such a priori bounds at our disposal, we can then harness the regularizing effect of the averaging operator $T^{-w}$ as discussed above to obtain a priori bounds that are robust even when considering singular potentials $b$ (refer to section \ref{a priori section}). The so obtained new a priori bounds for singular $b$ can then be used to obtain solutions to \eqref{strong formulation} using classical monotonicity arguments (refer to section \ref{passage to limit}). In summary, this allows us to prove our main theorem: 

\begin{theorem}[Existence of robustified solution]\label{main theorem}
Let $p > 1$ and $u_0 \in L^2(\Omega) \cap W^{1,p}_0(\Omega)$.
For $r\in [1,\infty)$ and $q\in  [r, \infty)$ let $b:\R^N\to \R^N$ satisfy $b\in L^{2q}(\mathbb{R}^N)$.  Suppose that $w:[0, T]\to \R^N$ is continuous and admits a local time $L$ which satisfies $L\in C^{0,\gamma}\big([0,T]; W^{1, r'}(\mathbb{R}^N)\big)$ for some $\gamma\in (1/2, 1)$. 

Then there exists a robustified solution
\begin{align*}
    u \in \left\{v \in L^\infty\big(0,T; L^2(\Omega) \cap W^{1,p}_0(\Omega) \big)\big| \, \partial_t v, \Div S(\nabla v) \in L^2([0,T] \times \Omega) \right\}
\end{align*}
to \eqref{regularized problem} in the sense of Definition \ref{def:robustified}. Moreover, the following a priori bound is valid
\begin{align}
\begin{aligned}
&\norm{u}_{L^\infty(0,T; L^2(\Omega))}^2 + \norm{\nabla u}_{L^\infty(0,T; L^p(\Omega))}^p + \norm{\partial_t u}_{L^2([0,T]\times \Omega)}^2 + \norm{\Div S(\nabla u)}_{L^2([0,T]\times \Omega)}^2 \\
&\hspace{3em} \lesssim \norm{u_0}_{L^2(\Omega)}^2+\norm{\nabla u_0}_{L^p(\Omega)}^p + \norm{b}_{L^{2q}(\mathbb{R}^N)}^4\norm{L}_{C^{0,\gamma}([0,T];W^{1, r'}(\mathbb{R}^N))}^2.
\end{aligned}
\end{align}
Suppose moreover that $b$ satisfies the monotonicity condition, for all $u,v \in \mathbb{R}^N$,
\begin{align} \label{eq:MonotoneB}
    (b(u)-b(v))\cdot ( u-v) \leq 0,
\end{align}
then robustified solutions to \eqref{regularized problem} in the sense of Definition \ref{def:robustified} are unique.
\end{theorem}
Let us illustrate this established regularization effect by means of a concrete example. A detailed verification of the claims of the Examples can be found in Appendix~\ref{appendix:Example}.

\begin{example} \label{ex:Main}
Let $K > 0$ and define the potential
\begin{align} \label{example b}
    b(u) := - |u|^{\eta-1}u \textbf{1}_{\{|u|\leq K\}}.
\end{align}
Then the following statements are true:
\begin{enumerate}
    \item \label{it:NonExistence} If $\eta \leq -1$, then~\eqref{Original problem}, i.e. the problem without regularizing path does not have a weak solution for $u_0=0$.
    \item \label{it:ExistenceRobustified} Let $\eta \in (-N/2,0)$ and
    \begin{align} \label{ex:ConditionH}
        H <\frac{1}{(2-4\eta) \vee N}
    \end{align}
    and $w^H$ be a $N$-dimensional fractional Brownian motion with Hurst parameter~$H$. Then for almost any realization $w^H(\omega):[0,T]\to \R^N$~\eqref{regularized problem} has a robustified solution for any $u_0\in L^2(\Omega)\cap W^{1, p}_0(\Omega)$.
\end{enumerate}
Essentially~\eqref{it:NonExistence} shows that in the unperturbed setting the origin is a singular state, which is due to the singularity of $b$ in zero. In contrast to this, the presence of the highly oscillating path $w^H(\omega)$ ensures that the solution $u$ does not spend too much time in the singularity of $b$. Condition~\eqref{ex:ConditionH} ensures that this effect is quantitatively sufficiently strong for the singularity not to obstruct existence theory, as formulated in our main theorem above. Overall, we can thus observe a regularization phenomenon in dimension $N\geq 3$. In particular, for $N=3$ and $\eta=1$, there is no weak solution to \eqref{Original problem} for $u_0=0$, while \eqref{regularized problem} with $u_0=0$ admits a robustified solution for almost every realization $w^H(\omega)$ of a fractional Brownian motion, provided $H<1/6$. 
\end{example}

\begin{rem}
Let us stress that we are employing a pathwise regularization by noise argument in the spirit of \cite{harang2020cinfinity} relying on the study of local times and their regularities and not in the spirit of \cite{gubicat}, \cite{galeati2020noiseless} based on the study of averaging operators. While results on the regularizing effect of averaging operators are slightly less optimal for realizations of fractional Brownian motion for example, they allow for a completely pathwise i.e. analytical treatment. In contrast, arguing in the spirit of \cite{gubicat}, \cite{galeati2020noiseless} would require to employ a tightness argument. Moreover, if the perturbing path is the realization of a stochastic process $w$, note that as our proof employs a compactness argument, we lose measurability of our solution with respect to the probability space on which $w$ is defined. While this can be retained in an approach along the lines of \cite{gubicat}, \cite{galeati2020noiseless}, one would only obtain a probabilistically weak solution in that case.

\end{rem}

\subsection*{Remarks on the literature}
Starting with the seminal work \cite{gubicat}, taken up again by \cite{galeati2020noiseless}, \cite{harang2020cinfinity}, \cite{gerencsergal}, \cite{tolomeo} pathwise regularization by noise has seen considerable developments in recent years. Areas in which such techniques have been succesfully implemented include particle systems \cite{particlesystems}, distribution dependent SDEs \cite{Galeati2022}, \cite{galeati2}, multiplicative SDEs \cite{galeatiharang}, \cite{bechtold}, \cite{toyomu}, \cite{dareiotis}, \cite{catellier2} and perturbed SDEs \cite{lukasz}. An extension of pathwise regularization techniques to the two parameter setting with applications to regularization by noise for the stochastic wave equation was established in \cite{bechtold3}. Pathwise regularization by noise for the multiplicative stochastic heat equations with spatial white noise were treated in \cite{catellierharang}, \cite{bechtold2}. 

\subsection*{Outline of the paper}
In Section~\ref{sec:Setup} we introduce basic notation and the notion of solution. Section~\ref{strong solutions} addresses quantified but implicit a priori bounds for strong solutions. In Section~\ref{a priori section} we close the a priori bounds from Section~\ref{strong solutions}. Lastly, Section~\ref{passage to limit} copes with the identification of the limit and a discussion on uniqueness. In the Appendix~\ref{app:Appendix} we verify Example~\ref{ex:Main}, present results related to occupation measures, local times and sewing techniques and give details on the identification of limits for monotone equations.

\section{Mathematical setup} \label{sec:Setup}

 Let $\Omega \subset \R^d$ for $d \geq 1$ be a bounded Lipschitz domain. For some given $T>0$ we denote by $I := [0,T]$ the time interval and write $\Omega_T := I \times \Omega$ for the time space cylinder. We write $f \lesssim g$ for two non-negative quantities $f$ and $g$ if $f$ is bounded by $g$ up to a multiplicative constant. Accordingly we define $\gtrsim$ and $\eqsim$. Moreover, we denote by $c$ a generic constant which can change its value from line to line. For $r\in [1, \infty]$, we denote by $r' = r/(r-1)$ its H\"older conjugate. We do not distinguish between scalar, vector and matrix-valued functions.

\subsection*{Function spaces} \label{sec:Function spaces}
As usual, for $q \in [1,\infty]$, let $L^q(\Omega)$ denote the Lebesgue space and $W^{1,q}(\Omega)$ the Sobolev space on the domain~$\Omega$, respectively.  
Furthermore, we denote by $W^{1,q}_0(\Omega)$ the Sobolev space with zero boundary values. For $q < \infty$ it is the closure of $C^\infty_c(\Omega)$ (smooth functions with compact support) in the $W^{1,q}(\Omega)$-norm. Addtionally, we denote by $W^{-1,q'}(\Omega)$ the dual of $W^{1,q}_0(\Omega)$. We abbreviate function spaces on the domain by $L^p_x := L^p(\Omega)$ and on the full space by $L^p:=L^p(\R^N)$ with suitable modifications for Sobolev norms. The inner product in $L^2_x$ is denoted by~$(\cdot, \cdot )$ and duality pairings are written as $\langle \cdot, \cdot \rangle$. 

For a Banach space $\left(X, \norm{\cdot}_X \right)$ let $L^q(I;X)$ be the Bochner space of Bochner-measurable functions $u: I \to X$ satisfying $t \mapsto \norm{u(t)}_X \in L^q(I)$. Moreover, $C^0(I;X)$ is the space of continuous functions with respect to the norm-topology. We also use $C^{0,\alpha}(I;X)$, $\alpha \in (0,1]$, for the space of $\alpha$-H\"older continuous functions.  We abbreviate the notation $ L^q_t X := L^q(I;X) $ and $C^0_t X = C^0(I;X)$. 

For $s\in \R^+$ we further note by by $H^{s}$ the space of Bessel-potentials, 
\[
H^{s}:=\{ f\in \mathcal{S}'|\ \norm{f}_{H^{s}}:=\norm{\mathcal{F}^{-1} (1+|\xi|^2)^{s/2}\mathcal{F}f}_{L^2}<\infty \}
\]
Let us also recall a particular instance of Young's convolution inequality adapted to our setting, i.e. 
\begin{equation}
\label{young}
    \norm{f*g}_{C^{0,1}}\lesssim \norm{f}_{L^r}\norm{g}_{W^{1,r'}},
\end{equation}
which is a consequence of $D_x (f*g)=f*(D_xg)$ and Young's convolution inequality in the classical setting.

\subsection*{Solution concepts}
In the following, let us discuss different notions of solutions to \eqref{regularized problem}. We begin with the classical notions of weak and strong solutions before passing on to so called robustified solutions that exploit the gain in regularity due to the regularizing path $w$ as discussed in the introduction. 
\begin{definition}[Classical] \label{def:ClassicalSolution}
A function $u$ is called weak solution to~\eqref{regularized problem} if
\begin{enumerate}
    \item (Regularity) $u \in C^0_{t} L^2_x \cap L^{p}_t W^{1,p}_{0,x}$, $b(u_\cdot - w_\cdot) \in L^{p'}_t W^{-1,p'}_x$ and 
    \item (Tested equation) for all $t \in I$ and $\xi \in C^\infty_{c,x}$
    \begin{align} \label{def:WeakSolutionEq}
    \begin{aligned}
        &\int_\Omega (u_t - u_0) \cdot \xi \dd x + \int_0^t \int_\Omega S(\nabla u) : \nabla \xi \dd x \dd s = \int_0^t \left \langle b(u - w), \xi \right\rangle \dd s.
    \end{aligned}
    \end{align}
\end{enumerate}

A weak solution $u$ is called strong solution to~\eqref{regularized problem} if additionally
\begin{enumerate}
    \item (Regularity)  $\partial_t u, \, \Div S(\nabla u), \, b(u_\cdot - w_\cdot) \in L^{2}_t L^2_x$ and 
    \item (Point-wise equation) for almost all $(t,x) \in \Omega_T$
    \begin{align} \label{def:StrongSolutionEq}
    \begin{aligned}
         \partial_t u - \Div S(\nabla u) = b(u- w).
    \end{aligned}
    \end{align}
\end{enumerate}
\end{definition}

\begin{definition}[Robustified] \label{def:robustified}
We call $u$ a solution to \eqref{regularized problem} in the robustified sense if 
\begin{align}
    u \in \left\{v \in C^{0, 1/2}_tL^2_x \cap L^\infty_tW^{1,p}_{0,x}  \big| \, \partial_t v, \Div S(\nabla v) \in L^2_t L^2_x \right\}
\end{align}
and if for any $t\in I$ we have
\begin{equation} \label{eq:robustified}
       u_t-u_0-\int_0^t \Div S(\nabla u_r) \dd r=(\mathcal{I}A^u)_{0,t}.
\end{equation}
understood as an equality in $L^2_x$, where $\mathcal{I}A^u$ denotes the sewing\footnote{Refer to Lemma \ref{sewing} in the Appendix} of the germ
\[
A^u_{s,t}=(b*L_{s,t})(u_s).
\]
\end{definition}

The next two lemmata verify that the concept of robustified solutions coincides with classically defined strong solutions in the smooth setting.
\begin{lemma}
\label{identification}Let $b$ be smooth and bounded and assume that $w$ admits a local time $L\in C^{0, 1/2+\epsilon}_tL^2_x$ for some $\epsilon>0$. Then any strong solution $u\in C^{0,1/2}_tL^2_x$ to \eqref{regularized problem} is a solution in the robustified sense. 
\end{lemma}
\begin{proof}
Since $u$ is a strong solution, it suffices to show that for any $s<t\in [0,T]$ we have
\[
(\mathcal{I}A^u)_{s,t}=\int_s^t b(u_r-w_r) dr
\]
Remark first that $A^u$ does admit a sewing with values in $L^2_x$. Indeed,
\begin{align*}
   \norm{(\delta A^u)_{svt}}_{L^2_x}&=\norm{ (b*L_{v,t})(u_s)-(b*L_{v,t})(u_v)}_{L^2_x}\\
   &\lesssim \norm{b*L_{v,t}}_{C^{0,1}}\norm{u_s-u_v}_{L^2_x}\\
   &\lesssim \norm{b}_{H^1}\norm{L}_{C^{0,1/2+\epsilon}_t L^2}|t-s|^{1/2+\epsilon}\norm{u}_{C^{0,1/2}_t L^2_x}|t-s|^{1/2}
\end{align*}
for some $\epsilon>0$,  meaning we may apply the Sewing Lemma \ref{sewing}. Moreover, we have 
\begin{align*}
    &\norm{A^u_{s,t}-\int_s^tb(u_r-w_r)dr}_{L^2_x}=\norm{\int_s^t b(u_s-w_r)-b(u_r-w_r)dr}_{L^2_x}\\
    &\leq \norm{b}_{C^{0,1}}\norm{u}_{C^{0,1/2}_tL^2_x}\int_s^t|r-s|^{1/2}dr\lesssim  \norm{b}_{C^{0,1}}\norm{u}_{C^{0,1/2}_tL^2_x}|t-s|^{3/2},
\end{align*}
from which we conclude that 
\begin{align*}
    &\norm{(\mathcal{I}A^u)_{s,t}-\int_s^t b(u_r-w_r) dr}_{L^2_x}\\
    \leq& \norm{A^u_{s,t}-(\mathcal{I}A^u)_{s,t}}_{L^2_x}+\norm{A^u_{s,t}-\int_s^tb(u_r-w_r)dr}_{L^2_x}= O(|t-s|^{1+\epsilon}).
\end{align*}
Hence the function 
\[
t\to (\mathcal{I}A^u)_{0,t}-\int_0^tb(u_r-w_r)dr
\]
is constant. As it starts in zero, this concludes the claim. 
\end{proof}
\begin{lemma} \label{lem:Rob2Strong}
Let $b$ be smooth and bounded and let $w\in C^{0,\alpha}_t$ for some $\alpha > 0$ be a  path with local time $L \in C^{0,1/2+\epsilon}_t L^2_x$ for some $\epsilon> 0$. Assume $u$ to be a robustified solution, then

\begin{enumerate}
    \item \label{item:1} the sewing is weakly differentiable, i.e.,
\begin{align*}
 t\mapsto (\mathcal{I}A^u)_{t} \in W^{1,2}(0,T; L^2_x),
\end{align*}
and moreover, we find
\begin{align*}
    \partial_t \big[ (\mathcal{I}A^u)\big] \big|_t &:= \lim_{h\to 0} h^{-1}\big( (\mathcal{I}A^u)_{t+h} - (\mathcal{I}A^u)_{t} \big) = b(u_t - w_t).
\end{align*}

\item \label{item:2}
$u$ is a strong solution and for almost all $t \in I$
\begin{align*}
  \partial_t u - \Div S(\nabla u) = b(u - w)
\end{align*}
as an equation in $L^2_x$.
\end{enumerate}
\end{lemma}

\begin{proof}
First we will address~\eqref{item:1}. We will not use that $u$ is a robustified solution and only exploit $u\in C^{0, 1/2}_tL^2_x$ but trace the dependents of the sewing and the regularity of the germ. In principle the statement of~\eqref{item:1} is of independent interest.

Recall that $(\mathcal{I}A^u)$ is constructed as the germ 
\begin{align*}
    A_{s,t} = \int_s^t b(u_s - w_\tau) \dd \tau = \big( b * L_{s,t} \big) (u_s).
\end{align*}
Additionally, we have as in the previous Lemma 
\begin{align*}
    \norm{(\mathcal{I}A^u)_{s,t} - A_{s,t}}_{L^2_x} \lesssim \norm{b}_{H^1} \norm{L}_{C^{0,1/2+\epsilon}_t L^2} \norm{u}_{C^{0,1/2}_t L^2_x} \abs{t-s}^{1+\epsilon}.
\end{align*}
On the other hand
\begin{align*}
   \norm{\frac{A_{t,t+h}}{h} -b(u_{t} - w_t)}_{L^2_x} &= \norm{h^{-1} \int_{t}^{t+h} b(u_{t} - w_\tau) -  b(u_t - w_t) \dd \tau}_{L^2_x} \\
   &\leq \norm{  h^{-1} \int_{t}^{t+h}   \seminorm{b}_{C^{0,1}}\abs{w_\tau - w_t}\dd \tau }_{L^2_x} \lesssim \seminorm{b}_{C^{0,1}} \seminorm{w}_{C^{0,\alpha}_t} \abs{h}^{\alpha}.
\end{align*}
Overall, division by $h$ and passing with $h \to 0$ implies thus
\begin{align*}
    \norm{\frac{(\mathcal{I}A^u)_{t+h} - (\mathcal{I}A^u)_{t}  }{h} -b(u_t-w_t)}_{L^2_x} \rightarrow  0,
\end{align*}
establishing the first claim. 
Next, we take a look at~\eqref{item:2}. By part~\eqref{item:1} the right hand side of~\eqref{eq:robustified} is weakly differentiable in $L^2_x$ and due to the regularity assumptions $\partial_t u \in L^2_tL^2_x$ and $\Div S(\nabla u) \in L^2_t L^2_x$ also the left hand side is differentiable. Therefore we may rescale~\eqref{eq:robustified} by $\abs{t-s}^{-1}$ and pass to the limit $t \to s $ to obtain
\begin{align*}
    \partial_t u - \Div S(\nabla u) = b(u-w).
\end{align*}

\end{proof}
\begin{rem}
Note that the key difference between strong and robustified solutions lies in the fact that the latter can exploit the regularization property of the local time $L$ associated with $w$. In particular, provided this local time is sufficiently regular, the definition is meaningful even in instances in which $b$ only enjoys distributional regularity as discussed in the introduction. 
\end{rem}

\section{Classical a priori bounds for strong solutions}
\label{strong solutions}
Singular potentials $b$ that do not satisfy the above growth conditions are in general inaccessible to the classical theory as seen above. However, if we restrict ourselves to regularized approximations of $b$, i.e., we assume that there exists $(b_\varepsilon)_{\varepsilon \in (0,1)} \in C^{0,1}$ such that $b_\varepsilon \to b \in L^{2q}$, then for each $\varepsilon \in (0,1)$ the classical theory is applicable and existence of strong solutions to~\eqref{regularized problem} for smooth $b$ is a classical result. It can be found e.g. in the books~\cite{MR0259693,MR0636412,MR1033498}. The objective of the present section is to trace the precise form of the a priori estimates that will be robustified in the next section. 
\begin{theorem} \label{thm:Strong solutions}
Let $p \in (1,\infty)$, $u_0^\varepsilon \in L^2_x \cap W^{1,p}_{0,x}$ and $b_\varepsilon \in C^{0,1}$. Then there exists a unique solution $u^\varepsilon$ solving~
\begin{alignat}{2}\label{regularized problem II}
\begin{aligned}
    \partial_t u^\epsilon-\Div S(\nabla u^\epsilon) &=b_\epsilon(u^\epsilon-w) &&\quad \text{ on } \Omega\times [0,T]\\
        u^\epsilon|_\Omega&=0 &&\quad \text{ on }  \partial\Omega\times [0,T]\\
        u^\epsilon(0, \cdot)&=u_0^\varepsilon &&\quad \text{ on }   \Omega.
\end{aligned}
\end{alignat}
in the strong sense of Definition~\ref{def:ClassicalSolution}.
Moreover, the following a priori bounds are valid
\begin{align}\label{eq:Strong solutionsA}
\begin{aligned}
 \sup_{t \leq T} \frac{1}{p} \norm{ \nabla u_t^\varepsilon}_{L^p_x}^p &+    \int_0^T \norm{\partial_t u_t^\varepsilon}_{L^2_x}^2 + \norm{\Div S(\nabla u^\varepsilon_t)}_{L^2_x}^2 \dd t \\
   &\leq  \frac{3}{p} \norm{\nabla u_0^\varepsilon}_{L^p_x}^p + 3 \int_0^T \int_\Omega \abs{b_\varepsilon(u_t^{\varepsilon}(x) - w_t)}^2 \dd x \dd t,
   \end{aligned}
\end{align}
and
\begin{align}\label{eq:Strong solutionsB}
   \sup_{t\leq T} \frac{1}{4} \norm{u_t^{\varepsilon}}_{L^2_x}^2 + \int_0^T \norm{\nabla u_t^{\varepsilon}}_{L^p_x}^p \dd t \leq  \norm{u_0^{\varepsilon}}_{L^2_x}^2 + 2 T \int_0^T \int_\Omega \abs{b_\varepsilon(u_t^{\varepsilon}(x) - w_t)}^2 \dd x \dd t.
\end{align}
\end{theorem}

\begin{proof}
The existence of a unique strong solution $u^\varepsilon$ to~\eqref{regularized problem II} is standard, see e.g.~\cite{MR0259693,MR0636412,MR1033498}, so let us just recall the main steps employed: In order to obtain existence, one first performs a Galerkin projection to the problem. The existence of solutions to the so obtained finite dimensional problem is done through a fixed point theorem. Next, using monotonicity of the $p$-Laplace operator, one establishes a priori bounds uniformly along solutions to the projected problems. By Banach-Alaoglu, one extracts a weak-* convergent subsequence, whose limit one has to identify as a solution to the problem. Identifying the limit in the non-linearity $b$ is done thanks to the Aubin-Lions Lemma \ref{lem:Aubin-Lions}. Identifying the limit in the $p$-Laplace operator is done with Minty's Lemma \ref{minty}. Finally, uniqueness is obtained by monotonicity of the $p$-Laplace operator.

We will argue on the weak~\eqref{eq:Strong solutionsB} and strong~\eqref{eq:Strong solutionsA} energy estimates separately.

The weak energy estimate~\eqref{eq:Strong solutionsB} naturally occurs when multiplying~\eqref{regularized problem II} by $u^\varepsilon$. Integration in space and integration by parts imply
\begin{align*}
    \partial_t \left( \frac{1}{2} \norm{u_t^\varepsilon}_{L^2_x}^2 \right) + \norm{\nabla u_t^\varepsilon}_{L^p_x}^p = \int_{\Omega}  b_\varepsilon(u_t^\varepsilon(x) - w_t) \cdot u_t^\varepsilon(x) \dd x.
\end{align*}
Integration in time, together with H\"older's and Young's inequalities result in
\begin{align}
\begin{aligned} \label{eq:EstimateWeak01}
    \frac{1}{2} \norm{u_t^\varepsilon}_{L^2_x}^2 &+ \int_0^t \norm{\nabla u_s^\varepsilon}_{L^p_x}^p \dd s =  \frac{1}{2} \norm{u_0^\varepsilon}_{L^2_x}^2 + \int_0^t \int_{\Omega}  b_\varepsilon(u_s^\varepsilon(x) - w_s) \cdot u_s^\varepsilon(x) \dd x \dd s \\
    & \leq   \frac{1}{2} \norm{u_0^\varepsilon}_{L^2_x}^2 + t\int_0^t \int_{\Omega}  \abs{b_\varepsilon(u_s^\varepsilon(x) - w_s)}^2 \dd x \dd s + \frac{1}{4t} \int_0^t \norm{u_s^\varepsilon}_{L^2_x}^2  \dd s.
    \end{aligned}
\end{align}
Finally, take the supremum in $t$ over $(0,T)$ to conclude
\begin{align*}
    \sup_{t \in (0,T)} \frac{1}{4} \norm{u_t^\varepsilon}_{L^2_x}^2 
    \leq \frac{1}{2} \norm{u_0^\varepsilon}_{L^2_x}^2 + T\int_0^T \int_{\Omega}  \abs{b_\varepsilon(u_s^\varepsilon(x) - w_s)}^2 \dd x \dd s.
\end{align*}
This and~\eqref{eq:EstimateWeak01} establish the inequality~\eqref{eq:Strong solutionsB}.

The strong energy estimate follows from squaring both sides of~\eqref{regularized problem II} and integration in space and time
\begin{align*}
    \int_0^t \int_{\Omega} \abs{\partial_t u_s^\varepsilon - \Div S(\nabla u_s^\varepsilon)}^2 \dd x \dd s = \int_0^t \int_{\Omega} \abs{b_\varepsilon(u_s^\varepsilon(x) - w_s)}^2 \dd x \dd s.
\end{align*}
Note that, due to integration by parts,
\begin{align*}
   &-2 \int_0^t \int_{\Omega}  \partial_t \cdot u_s^\varepsilon \,\Div S(\nabla u_s^\varepsilon) \dd x \dd s = 2 \int_0^t \int_{\Omega}  \partial_t \nabla u_s^\varepsilon : S(\nabla u_s^\varepsilon) \dd x \dd s \\
   &\hspace{2em} = \frac{2}{p} \int_0^t \partial_t \norm{\nabla u_s^\varepsilon}_{L^p_x}^p \dd s = \frac{2}{p} \left(\norm{\nabla u_t^\varepsilon}_{L^p_x}^p - \norm{\nabla u_0^\varepsilon}_{L^p_x}^p\right).
\end{align*}
Therefore, we obtain
\begin{align*}
    &\frac{2}{p} \norm{\nabla u_t^\varepsilon}_{L^p_x}^p + \int_0^t \norm{\partial_t u_s^\varepsilon}_{L^2_x}^2 +\norm{ \Div S(\nabla u_s^\varepsilon)}_{L^2_x}^2  \dd s \\
    &\hspace{2em} =\frac{2}{p} \norm{\nabla u_0^\varepsilon}_{L^p_x}^p + \int_0^t \int_{\Omega} \abs{b_\varepsilon(u_s^\varepsilon(x) - w_s)}^2 \dd x \dd s.
\end{align*}
The claim~\eqref{eq:Strong solutionsA} immediately follows after taking the supremum in $t$ over $(0,T)$.
\end{proof}

\begin{rem}
The weak energy estimate~\eqref{eq:Strong solutionsB} can be generalized to less restrictive assumptions on~$b_\varepsilon$, e.g. it is sufficient to assume $b_\varepsilon(u^\varepsilon- w) \in \big( C_t^0 L^2_x \cap L^p_t W^{1,p}_{0,x} \big)'$.
\end{rem}

\begin{cor}\label{cor:HoelderReg}
Let the assumptions of Theorem~\ref{thm:Strong solutions} be satisfied. Then
\begin{align} \label{eq:HoelderReg}
\norm{u^\varepsilon}_{C^{0,1/2}_t L^2_x}^2 \leq 4 \norm{u_0^\varepsilon}_{L^2_x}^2 +  \frac{3}{p} \norm{\nabla u_0^\varepsilon}_{L^p_x}^p + (3 + 8T) \int_0^T \int_\Omega \abs{b_\varepsilon(u_t^{\varepsilon}(x) - w_t)}^2 \dd x \dd t.
\end{align}
\end{cor}
\begin{proof}
Due to~\eqref{eq:Strong solutionsB} it holds
\begin{align}\label{eq:Inf1}
    \norm{u^\varepsilon}_{L^\infty_t L^2_x}^2 \leq 4 \norm{u_0^\varepsilon}_{L^2_x} + 8T \int_0^T \int_\Omega \abs{b_\varepsilon( u_t^\varepsilon(x) - w_t)}^2 \dd x \dd t.
\end{align}
The fundamental theorem of calculus and H\"older's inequality reveal
\begin{align*}
    \norm{u_t^\varepsilon - u_s^\varepsilon}_{L^2_x}^2 = \norm{\int_s^t \partial_t u_r^\varepsilon \dd r}_{L^2_x}^2 \leq \abs{t-s} \int_s^t \norm{\partial_t u_r^\varepsilon}_{L^2_x}^2 \dd r.
\end{align*}
The estimate~\eqref{eq:Strong solutionsA} bounds
\begin{align}\label{eq:Inf2}
    \seminorm{u^\varepsilon}_{C^{0,1/2}_t L^2_x}^2 \leq \frac{3}{p} \norm{\nabla u_0^\varepsilon}_{L^p_x}^p + 3 \int_0^T \int_\Omega \abs{b_\varepsilon(u_t^{\varepsilon}(x) - w_t)}^2 \dd x \dd t.
\end{align} 
Adding~\eqref{eq:Inf1} and~\eqref{eq:Inf2} verifies the claim.
\end{proof}

%\section{Robustified formulation}

\section{Robustified a priori bounds for the mollified problem}
\label{a priori section}
Given a singular potential $b\in L^{2q}$, we have seen in the previous section that for a mollification $b_\epsilon=b*\rho_\epsilon$, we obtain a unique strong solution $u^\varepsilon$ to
\begin{equation}
    \begin{split}
        \partial_t u^\epsilon- \Div S(\nabla u^\epsilon)  = b_\epsilon(u^\epsilon-w).
        \label{a priori start}
    \end{split}
\end{equation}

 In the following, we show that in harnessing the regularizing effect due to the perturbing path $w$, robust a priori bounds uniformly in $\epsilon>0$ can be obtained. More precisely, we establish the a priori bound
\[
\norm{u^\epsilon}_{L^\infty_t W^{1, p}_x}^p+\norm{u^\epsilon}_{C^{0, 1/2}_tL^2_x}^2\leq C_{T, \Omega}(\norm{u_0}_{L^2_x}^2+\norm{\nabla u_0}_{L^p_x}^p +\norm{b}_{L^{2q}}^4\norm{L}_{C^{0, \gamma}_tW^{1, r'}}^2)
\]
uniformly in $\epsilon>0$. Towards this end, we first show that the right-hand side of \eqref{eq:Strong solutionsA}, \eqref{eq:Strong solutionsB} and \eqref{eq:HoelderReg} can be robustified in the sense of the following identification. 
\begin{lemma}
\label{identificationII} Let $r\in [1,\infty)$ and $\gamma > 1/2$. Suppose $w:[0, T]\to \R^N$ is continuous and admits a local time $L$ such that $L\in C^{0, \gamma}_tW^{1, r'}$ and  $b\in L^{2q}$ for $q\in [r, \infty)$. Then for any $\epsilon>0$ fixed, we have
\begin{equation}
\int_0^T\int_\Omega |b_\epsilon(u^\epsilon_r-w_r)|^2 \dd x \dd r=(\mathcal{I}A^\epsilon)_{0,T}
    \label{identificationIII}
\end{equation}
where $\mathcal{I}$ denotes the sewing of the germ
\[
A^\epsilon_{s,t}=\int_\Omega (b_\epsilon^2*L_{s,t})(u^\epsilon_s) \dd x,
\]
and where we used the shorthand notation $b_\epsilon^2(u):= |b_\epsilon(u)|^2$.  Moreover, there holds the a priori bound
\[
|(\mathcal{I}A^\epsilon)_{s,t}|\leq C_\Omega \norm{b}_{L^{2q}}^2\norm{L}_{C^{0, \gamma}_tW^{1, r'}}(1+\norm{u^\epsilon}_{C^{0, 1/2}_tL^2_x})|t-s|^\gamma.
\]
 
\end{lemma}
\begin{proof}
Recall that by Corollary \ref{cor:HoelderReg}, we have $u^\epsilon\in C^{0, 1/2}_tL^2_x$ for $\epsilon>0$ fixed. The first part of the claim is established similarly to Lemma \ref{identification}, the main difference being that we exploit the regularity gained from the local time $L$ in order to obtain the a priori bound in the second part of the claim. Let us start by remarking that for $\epsilon>0$ fixed, $A^\epsilon$ does indeed admit a sewing as 
\begin{align*}
    |\delta A^\epsilon_{s,r,t}|&\leq \int_\Omega | (b_\epsilon^2*L_{r,t})(u^\epsilon_s)- (b_\epsilon^2*L_{r,t})(u^\epsilon_r)| \dd x\\
    &\leq \norm{b_\epsilon^2*L_{r,t}}_{C^{0,1}}\norm{u^\epsilon_r-u^\epsilon_s}_{L^1_x}\\
    &\lesssim \norm{b_\epsilon^2}_{L^q}\norm{L_{r,t}}_{W^{1, q'}}\norm{u^\epsilon}_{C^{0, 1/2}_tL^2_x}|r-s|^{1/2}\\
    &\lesssim \norm{b}_{L^{2q}}^2\norm{L}_{C^{\gamma}_tW^{1, r'}}\norm{u^\epsilon}_{C^{0, 1/2}_tL^2_x}|r-s|^{1/2}|t-r|^{\gamma},
\end{align*}
where we exploited that due to continuity of $w$, $(L_t)_t$ is of compact support uniformly in $t\in [0,T]$ and thus $\norm{L_t}_{W^{1, q'}}\lesssim \norm{L_t}_{W^{1, r'}}$ . Moreover, note that for $\epsilon>0$ fixed, we have
\begin{align*}
    \left| A^\epsilon_{s,t}-\int_s^t\int_\Omega b_\epsilon^2(u^\epsilon_r-w_r) \dd x \dd r \right|&=\left|\int_s^t \int_\Omega b^2_\epsilon(u^\epsilon_s-w_r)-b^2_\epsilon(u^\epsilon_r-w_r) \dd x \dd r\right|\\
    &\lesssim \norm{b_\epsilon^2}_{C^1}\norm{u^\epsilon}_{C^{0, 1/2}_tL^2_x}\int_s^t|r-s|^{1/2} \dd r\\
    &=\norm{b_\epsilon^2}_{C^1}\norm{u^\epsilon}_{C^{0, 1/2}_tL^2_x}|t-s|^{3/2}.
\end{align*}
Similar to Lemma \ref{identification}, we  can thus conclude that indeed
\[
\int_0^T\int_\Omega b_\epsilon^2(u^\epsilon_r-w_r) \dd x \dd r=(\mathcal{I}A^\epsilon)_{0,T}
\]
for any $\epsilon>0$ fixed. Moreover, exploiting the a priori bounds that come with the Sewing Lemma \ref{sewing}, we infer 
\[
|(\mathcal{I}A^\epsilon)_{s,t}|\leq |A^\epsilon_{s,t}|+|(\mathcal{I}A^\epsilon)_{s,t}-A^\epsilon_{s,t}|\lesssim  \norm{b}_{L^{2q}}^2\norm{L}_{C^{\gamma}_tW^{1, r'}}(1+\norm{u^\epsilon}_{C^{0, 1/2}_tL^2_x}).
\]
where we used that 
\[
|A^\epsilon_{s,t}|\lesssim \norm{b^2_\epsilon*L_{s,t}}_{L^\infty}\lesssim \norm{b}_{L^{2q}}^2\norm{L}_{C^{0, \gamma}_tL^{r'}}|t-s|^\gamma,
\]
which completes the proof. 
\end{proof}
\begin{rem}
The sewing enables the local time to regularize the interplay of $b_\varepsilon$ and $u^\varepsilon$. In particular the square $\abs{b_\varepsilon(u^\varepsilon - w)}^2$ only acts on $b_\varepsilon$ and not on $u^\varepsilon$. Therefore, it is possible to work in the $L^1_x$ framework for $u^\varepsilon$. Indeed, a short inspection of the proof of Lemma~\ref{identificationII} shows that H\"older regularity of $u^\varepsilon$ as a $L^1_x$-valued function is sufficient for the identification. Moreover, upon replacing $u^\epsilon$ by a generic function $v\in C^\alpha_tL^1_x$, the identification \eqref{identificationIII} holds provided $\alpha+\gamma>1$. 
\end{rem}

\begin{cor} \label{cor:UnifiedBounds}
 Let $r\in [1,\infty)$ and $\gamma > 1/2$. Suppose $w:[0, T]\to \R^N$ is continuous and admits a local time $L$ such that $L\in C^{0, \gamma}_tW^{1, r'}$ and  $b\in L^{2q}$ for $q\in [r, \infty)$. Let $u^\epsilon$ be the unique solution to \eqref{regularized problem II} of Theorem \ref{thm:Strong solutions} associated to the mollification $b_\epsilon$ of $b$. Then we have the a priori bound 
\begin{equation}
\begin{split}
& \norm{ \nabla u^\varepsilon}_{L^\infty_t L^p_x}^p + \norm{u^\varepsilon}_{C^{0,1/2}_t L^2_x}^2 \lesssim \norm{u_0}_{L^2_x}^2+\norm{\nabla u_0}_{L^p_x}^p + \norm{b}_{L^{2r}}^4\norm{L}_{C^{0, \gamma}_tW^{1, r'}}^2.
    \label{a priori II}
\end{split}
\end{equation}
\end{cor}
\begin{proof}
Plugging the a priori bound from the above Lemma \ref{identificationII} back into Corollary \ref{cor:HoelderReg} we obtain
\begin{align*} 
      \norm{u^\epsilon}^2_{C^{0, 1/2}_tL^2_x}&\leq 4 \norm{u_0}_{L^2_x}^2 +  \frac{3}{p} \norm{\nabla u_0}_{L^p_x}^p + (3 + 8T) \int_0^T \int_\Omega \abs{b_\varepsilon(u_t^{\varepsilon}(x) - w_t)}^2 \dd x \dd t\\ 
      &\lesssim \norm{u_0}_{L^2_x}^2+\norm{\nabla u_0}_{L^p_x}^p+C_{T, \Omega}\norm{b}_{L^{2r}}^2\norm{L}_{C^{0, \gamma}_tW^{1, r'}}(1+\norm{u^\epsilon}_{C^{0, 1/2}_tL^2_x}).
\end{align*}
An application of Lemma~\ref{app:Algebraic} with $a =  \norm{u^\epsilon}_{C^{0, 1/2}_tL^2_x}$, $K= \norm{u_0}_{L^2_x}^2+\norm{\nabla u_0}_{L^p_x}^p+C_{T, \Omega}\norm{b}_{L^{2q}_x}^2\norm{L}_{C^{0, \gamma}_tW^{1, r'}_x}$ and $C =C_{T, \Omega}\norm{b}_{L^{2q}_x}^2\norm{L}_{C^{0, \gamma}_tW^{1, r'}_x}$ results in
\begin{equation}
    \norm{u^\epsilon}^2_{C^{0, 1/2}_tL^2_x} \lesssim \norm{u_0}_{L^2_x}^2+\norm{\nabla u_0}_{L^p_x}^p + C_{T, \Omega}^2\norm{b}_{L^{2q}}^4\norm{L}_{C^{0, \gamma}_tW^{1, r'}}^2
\label{a priori I}
\end{equation}
uniformly in $\epsilon>0$. 

Note that similarly due to \eqref{eq:Strong solutionsA} in Theorem \ref{thm:Strong solutions} this also implies
\begin{align*}
& \norm{\nabla u^\varepsilon}_{L^\infty_t L^p_x}^p  \lesssim \norm{u_0}_{L^2_x}^2+\norm{\nabla u_0}_{L^p_x}^p + \norm{b}_{L^{2q}}^4\norm{L}_{C^{0, \gamma}_tW^{1, r'}}^2
\end{align*}
uniformly in $\epsilon>0$. 
\end{proof}

\begin{rem}[Refined a priori bounds]
We want to stress that even stronger a piori bounds than~\eqref{a priori II} are available. Indeed, the a priori bounds derived in Theorem~\ref{thm:Strong solutions} carry over to the robustified formulation, i.e., uniformly in $\varepsilon > 0$ it holds
\begin{align}
\begin{aligned}
&\norm{u^\varepsilon}_{L^\infty_t L^2_x}^2 + \norm{\nabla u^\varepsilon}_{L^\infty_t L^p_x}^p + \norm{\partial_t u^\varepsilon}_{L^2_t L^2_x}^2 + \norm{\Div S(\nabla u^\varepsilon)}_{L^2_t L^2_x}^2 \\
&\hspace{3em} \lesssim \norm{u_0}_{L^2_x}^2+\norm{\nabla u_0}_{L^p_x}^p + \norm{b}_{L^{2q}}^4\norm{L}_{C^{0, \gamma}_tW^{1, r'}}^2.
\end{aligned}
\end{align}
Therefore, we construct robustified solutions in the sense of Definition~\ref{def:robustified}. In particular, the right hand side of~\eqref{eq:robustified} is weakly differentiable in $L^2_x$. However, it is non trivial to identify the derivative. For smooth potentials an identification is possible as presented in Lemma~\ref{lem:Rob2Strong} \eqref{item:1}.
\end{rem}

\section{Passage to the limit}
\label{passage to limit}
In the following, we show that exploiting the a priori bounds obtained in the previous section allows a passage to the limit on the level of the robustified formulation. Towards this end, we will exploit monotonicity of the $p$-Laplace operator as well as suitable function space embeddings.\\
\\
Since~$\Omega$ is a bounded Lipschitz domain we can use the Rellich-Kondrachov theorem to conclude $W^{1,p}_{0,x} \hookrightarrow L^1_x$ compactly. Therefore, applying the Aubin-Lions Lemma~\ref{lem:Aubin-Lions} in the form of \eqref{eq:Aubin-Lions-cont} with $X=W^{1,p}_{0,x}$ and $B = Y = L^1_x$,  yields
\begin{align}
    L^\infty_t W^{1,p}_x \cap C^{0, 1/2}_t L^1_x \hookrightarrow\hookrightarrow C_tL^1_x.
\end{align}
Hence, the a priori bounds \eqref{a priori I} and \eqref{a priori II} also yield a subsequence still denoted $(u^\epsilon)_\epsilon$ converging strongly in $C_tL^1_x$. We denote the corresponding limit by $u$. In the following, we show that the strong convergence $u^\epsilon\to u$ in $C_tL^1_x$ allows to identify $u$ as a solution to \eqref{regularized problem} in the robustified sense of Definition \ref{def:robustified} meaning $u$ satisfies 
\[
u_t-u_s-\int_s^t \mbox{div} S(\nabla u_r) \dd r=(\mathcal{I}A^u)_{s,t}
\]
where $A^u_{s,t}=(b*L_{s,t})(u_s)$. In particular, we need to establish convergence of the corresponding sewing terms on the right hand side. This point will be addressed in the next Lemma. 
\begin{lemma}
 Let $r\in [1,\infty)$ and $\gamma > 1/2$. Suppose $w:[0, T]\to \R^N$ is continuous and admits a local time $L$ such that $L\in C^{0, \gamma}_tW^{1, r'}_x$ and  $b\in L^{2q}$ for $q\in [r, \infty)$. Let $u\in L^\infty_t W^{1,p}_x \cap C^{0, 1/2}_t L^1_x$ be as above, i.e. in particular such that $\norm{u^\epsilon-u}_{C_tL^1_x}\to 0$. Then, for
\[
A^u_{s,t}=(b*L_{s,t})(u_s), \qquad A^{u^\epsilon}_{s,t}=(b_\epsilon*L_{s,t})(u^\epsilon_s)
\]
we have $\norm{\mathcal{I}A^u-\mathcal{I}A^{u^\epsilon}}_{C^{0, \gamma}_tL^1_x}\to 0$. 
\end{lemma}
\begin{proof}
Note that 
\begin{align}
\begin{aligned} \label{eq:EstimateDiff}
   &\norm{A^u_{s,t}-A^{u^\epsilon}_{s,t}}_{L^1_x}\\
   &\hspace{2em} \leq\norm{(b-b_\epsilon)*L_{s,t}(u^\epsilon_s)}_{L^1_x}+\norm{(b*L_{s,t})(u^\epsilon_s)-(b*L_{s,t})(u_s)}_{L^1_x}\\
   &\hspace{2em} \lesssim\norm{b-b_\epsilon}_{L^{2q}}\norm{L_{s,t}}_{L^{(2q)'}}+\norm{b}_{L^{2r}}\norm{L_{s,t}}_{W^{1, (2r)'}}\norm{u^\epsilon-u}_{C_tL^1_x}\\
    &\hspace{2em} \lesssim \norm{b-b_\epsilon}_{L^{2q}}\norm{L_{s,t}}_{L^{r'}}+\norm{b}_{L^{2q}}\norm{L_{s,t}}_{W^{1, r'}}\norm{u^\epsilon-u}_{C_tL^1_x}\\
    &\hspace{2em} \lesssim |t-s|^{\gamma}\left(\norm{b-b_\epsilon}_{L^{2q}}\norm{L}_{C^{\gamma}_tL^{r'}}+\norm{b}_{L^{2q}}^2\norm{L}_{C^{0, 1/2}_tW^{1, r'}}\norm{u^\epsilon-u}_{C_tL^1_x}\right)
    \end{aligned}
\end{align}
where we exploited again that due to the compact support of $L$, we have $\norm{L}_{L^{(2r)'}}\lesssim\norm{L}_{L^{r'}}$ and similarly for the corresponding Sobolev scales.  Moreover, we have accordingly
\begin{align}
\begin{aligned} \label{eq:DeltaEstimate}
   \norm{ (\delta A^{u^\epsilon})_{s,r,t} }_{L^1_x}&=\norm{(b_\epsilon*L_{r,t})(u^\epsilon_r)-(b_\epsilon*L_{r,t})(u^\epsilon_s)}_{L^1_x}\\
   &\lesssim \norm{b_\epsilon}_{L^{2q}}\norm{L_{r,t}}_{W^{1, r'}}\norm{u^\epsilon_s-u^\epsilon_r}_{L^1_x}\\
   &\lesssim \norm{b}_{L^{2q}}\norm{L}_{C^{\gamma}_tW^{1, r'}}\norm{u^\epsilon}_{C^{0, 1/2}_tL^1_x}|t-s|^{\gamma+1/2}.
   \end{aligned}
\end{align}
The claim now follows from Lemma \ref{sewing convergence}.
\end{proof}
Identifying the limit in the other nonlinear term in the robustified formulation of definition \ref{def:robustified} is now a classical consequence from Minty's Lemma \ref{minty}. Overall, we have thus established the existence part of Theorem \ref{main theorem}.

\subsection{Uniqueness}
In this subsection we verify uniqueness as claimed in Theorem~\ref{main theorem}. In fact, we will prove a stronger result.
\begin{lemma}[Continuous dependence on initial data]
Let $r\in [1,\infty)$ and $q\in  [r, \infty)$. Moreover, let $b \in L^{2q}$ satisfy~\eqref{eq:MonotoneB} and $w:[0, T]\to \R^N$ be continuous and admit a local time $L$ which satisfies $L\in C^{0,\gamma}_t W^{1, r'}$ for some $\gamma\in (1/2, 1)$. Let~$u,v$ be two robustified solutions to~\eqref{regularized problem} started in $u_0, v_0 \in L^2_x$, respectively. 

Then for all $t \in [0,T]$ it holds
\begin{align} \label{eq:InitDep}
  \norm{u_t - v_t}_{L^2_x} \leq \norm{u_0 - v_0}_{L^2_x}.  
\end{align}
\end{lemma}

\begin{proof}
Let $u, v$ be two robustified solutions to~\eqref{regularized problem} in the sense of Definition \ref{def:robustified} starting in $u_0, v_0\in L^2_x$, respectively. In particular, they belong to the regularity class
\begin{align*}
    u, v\in \left\{z \in C^{0, 1/2}_tL^2_x \cap L^\infty_tW^{1,p}_{0,x}  \big| \, \partial_t z, \Div S(\nabla z) \in L^2_t L^2_x \right\}
\end{align*}
and satisfy the system of equations, for all $t\in I$ and almost all $x \in \Omega$,
\begin{subequations}\label{eq:System}
\begin{align} \label{eq:System01}
    u_t-u_0-\int_0^t\mbox{div} S(\nabla u_s) \dd s=(\mathcal{I}A^u)_{0,t},\\ \label{eq:System02}
    v_t-v_0-\int_0^t\mbox{div} S(\nabla v_s) \dd s=(\mathcal{I}A^v)_{0,t},
\end{align}
\end{subequations}
where $A^u_{s,t}=(b*L_{s,t})(u_s)$ and $A^v_{s,t}=(b*L_{s,t})(v_s)$, respectively.

Subtract~\eqref{eq:System02} from ~\eqref{eq:System01}, differentiate in time, multiply with $u_t - v_t$ and integrate in space to find
\begin{align}
\begin{aligned} \label{eq:LocalPresum}
    \left(\partial_t u_t - \partial_t v_t, u_t -v_t \right) &-  \left(\Div S(\nabla u_t) - \Div S(\nabla v_t), u_t -v_t \right) \\
    &= \left( \big(\partial_t \mathcal{I}A^u\big)_t -\big(\partial_t \mathcal{I}A^v\big)_t , u_t - v_t  \right).
    \end{aligned}
\end{align}
Notice that
\begin{align} \label{eq:Derivative01}
    2\left(\partial_t u_t - \partial_t v_t, u_t -v_t \right) = \partial_t \norm{u_t - v_t}_{L^2}^2.
\end{align}
Moreover, due to the monotonicity of the $p$-Laplace operator,
\begin{align} \label{eq:MonotonepLap}
    -  \left(\Div S(\nabla u_t) - \Div S(\nabla v_t), u_t -v_t \right) \geq 0.
\end{align}
Next, we integrate~\eqref{eq:LocalPresum} in time and use~\eqref{eq:Derivative01} and~\eqref{eq:MonotonepLap}
\begin{align}
    \norm{u_t - v_t}_{L^2}^2 \leq \norm{u_0 - v_0}_{L^2}^2 + 2\int_0^t  \left( \big(\partial_t \mathcal{I}A^u\big)_s -\big(\partial_t \mathcal{I}A^v\big)_s , u_s - v_s  \right) \dd s.
\end{align}
The estimate~\eqref{eq:InitDep} immediately follows provided we can verify
\begin{align} \label{eq:SufficientForUnique}
    \int_0^t  \left( \big(\partial_t \mathcal{I}A^u\big)_s -\big(\partial_t \mathcal{I}A^v\big)_s , u_s - v_s  \right) \dd s \leq 0.
\end{align}
Since we can not identify the time derivative of the sewing for non-smooth potentials, we will approximate the potential~$b$ by a sequence of smooth potentials~$(b^n)$ that preserve the monotonicity assumption~\eqref{eq:MonotoneB}. We will use Lemma~\ref{sewing convergence} to justify the convergence of the approximations and Lemma~\ref{lem:Rob2Strong}~\eqref{item:1} to identify the time derivative of the sewings for smooth potentials. 

Let $b^n=(\rho^n*b)$, where $(\rho^n)_n$ is a sequence of non-negative mollifiers. Notice that the monotonicity assumption on $b$, cf.~\eqref{eq:MonotoneB}, is preserved for $b^n$. Indeed,
\begin{align*}
    (b^n(u)-b^n(v)) &\cdot (u-v)\\
&=\int_{\R^N}\rho^n(z) (b(u-z)-b(v-z))\cdot ((u-z)-(v-z)) \dd z \leq  0.
\end{align*}

Due to integration by parts
\begin{align*}
    &\int_0^t \left( \partial_t \mathcal{I}A^u_s -  \partial_t \mathcal{I}A^v_s, u_s- v_s \right) \dd s \\
    &\hspace{3em} =   \left(  \mathcal{I}A^u_t -  \mathcal{I}A^v_t, u_t- v_t \right)  - \int_0^t \left( \mathcal{I}A^u_s -  \mathcal{I}A^v_s, \partial_t u_s- \partial_t v_s \right) \dd s \\
    &\hspace{3em} = \left(  \mathcal{I}A^u_t -  \mathcal{I}A^{n,u}_t, u_t- v_t \right) - \int_0^t \left( \mathcal{I}A^u_s -  \mathcal{I}A^{n,u}_s, \partial_t u_s- \partial_t v_s \right) \dd s \\
    &\hspace{4em} - \left(  \mathcal{I}A^v_t -  \mathcal{I}A^{n,v}_t, u_t- v_t \right) - \int_0^t \left( \mathcal{I}A^v_s -  \mathcal{I}A^{n,v}_s, \partial_t u_s- \partial_t v_s \right) \dd s \\
    &\hspace{4em} + \left(  \mathcal{I}A^{n,u}_t -  \mathcal{I}A^{n,v}_t, u_t- v_t \right)  - \int_0^t \left( \mathcal{I}A^{n,u}_s -  \mathcal{I}A^{n,v}_s, \partial_t u_s- \partial_t v_s \right) \dd s \\
    &\hspace{3em} =: \mathrm{R}_u^n + \mathrm{R}_v^n +\mathrm{R}_\mathrm{smooth}^n. 
\end{align*}
Here $A^{n,z}_{s,t} = (b^n * L_{s,t})(z_s)$, $z \in \{u,v\}$. 

By H\"older's inequality
\begin{align*}
    \sup_{t\in I} \abs{\mathrm{R}_u^n + \mathrm{R}_v^n} &\leq \left( \norm{\mathcal{I}A^u -  \mathcal{I}A^{n,u}}_{L^\infty_t L^2_x} + \norm{\mathcal{I}A^v -  \mathcal{I}A^{n,v}}_{L^\infty_t L^2_x} \right)\\
    &\hspace{4em} \left( \norm{u - v}_{L^\infty_t L^2_x} + \norm{\partial_t u - \partial_t v}_{L^1_t L^2_x} \right).
\end{align*}

Let $z \in \{u,v\}$. Similarly to~\eqref{eq:EstimateDiff} and~\eqref{eq:DeltaEstimate} it holds
\begin{align*}
    \norm{A^{n,z}_{s,t} - A^{z}_{s,t}}_{L^2_x} &\lesssim |t-s|^{\gamma} \norm{b-b^n}_{L^{2q}}\norm{L}_{C^{\gamma}_tL^{r'}}, \\
    \norm{ (\delta A^{n,z})_{srt}}_{L^2_x} &\lesssim  \norm{b}_{L^{2q}}\norm{L}_{C^{\gamma}_tW^{1, r'}}\norm{z}_{C^{0, 1/2}_tL^2_x}|t-s|^{\gamma+1/2}.
\end{align*}
Therefore, we can apply Lemma~\ref{sewing convergence} and obtain
\begin{align*}
    \norm{\mathcal{I}A^z -  \mathcal{I}A^{n,z}}_{C_t^0 L^2_x} \to 0 \quad \text{ as } n \to \infty.
\end{align*}
Thus,
\begin{align} \label{eq:Lim}
    \lim_{n\to \infty} \sup_{t\in I} \abs{\mathrm{R}_u^n + \mathrm{R}_v^n} = 0.
\end{align}

Reverting the partial integration, using Lemma~\ref{lem:Rob2Strong}~\eqref{item:1} and the monotonicity of~$b^n$
\begin{align}
\begin{aligned} \label{eq:MonotoneR}
    \mathrm{R}_\mathrm{smooth}^n &= \int_0^t \left( \partial_t \mathcal{I}A^{n,u}_s -  \partial_t \mathcal{I}A^{n,v}_s, u_s- v_s \right) \dd s \\
    &= \int_0^t \left( b^n(u_s - w_s) -  b^n(v_s - w_s), u_s- v_s \right) \dd s \leq 0.
    \end{aligned}
\end{align}
Finally,~\eqref{eq:SufficientForUnique} follows from~\eqref{eq:Lim} and~\eqref{eq:MonotoneR} and the assertion is verified.
\end{proof}

\section{Appendix} \label{app:Appendix}

\subsection*{Verification of Example~\ref{ex:Main}} \label{appendix:Example}
For convenience we recall~\eqref{example b}
\begin{align*}
    b(u)=-|u|^{\eta-1}u \textbf{1}_{\{|u|\leq K\}},
\end{align*}
where $\eta \in (-N/2,0)$ and $K > 0$.

\subsubsection*{Non-existence of weak solution}
First we verify that~\eqref{Original problem} does not have a weak solution started in~$0$ if $\eta \leq -1$. We proceed by contraposition. Assume that~\eqref{Original problem} does have a weak solution started in~$0$. Our goal is to use the energy equality~\eqref{eq:EstimateWeak01} to derive a contradiction.

Due to the density of smooth functions in $W^{-1,p'}$, we immediately find by~\eqref{def:WeakSolutionEq} 
\begin{align*}
    \norm{\partial_t u}_{L^{p'}_t W^{-1,p'}_x} &\leq \norm{S(\nabla u)}_{L^{p'}_t L^{p'}_x} + \norm{b(u)}_{L^{p'}_t W^{-1,p'}_x} \\
    &= \norm{\nabla u}_{L^{p}_t L^{p}_x}^{p-1} + \norm{b(u)}_{L^{p'}_t W^{-1,p'}_x}.
\end{align*}
Therefore~\eqref{def:WeakSolutionEq} implies for almost all $s \in [0,T]$ 
\begin{align} \label{eq:WeakDerivativesApp}
    \partial_t u_s - \Div \,S(\nabla u_s) = b(u_s) \quad \text{ in }W^{-1,p'}.
\end{align}
Moreover for all $u \in L^{p}_t W^{1,p}_{0,x} \cap W^{-1,p'}_t L^{p'}_x$ it holds
\begin{align} \label{eq:WeakDerivativesApp2}
    2\int_0^T \langle \partial_t u_s, u_s \rangle_{W^{-1,p'} \times W^{1,p}_0} \dd s = \norm{u_T}_{L^2}^2 - \norm{u_0}_{L^2}^2.
\end{align}
Since~$u\in L^p W^{1,p}_{0,x}$ we can test~\eqref{eq:WeakDerivativesApp} with $u_s$. Using~\eqref{eq:WeakDerivativesApp2} and~$u_0 = 0$, we arrive at
\begin{align} \label{eq:EnergyEqualityApp}
    \norm{u_T}_{L^2}^2  + 2\int_0^T \norm{\nabla u_r}_{L^p}^p \dd r + 2\int_0^T \int_{\Omega} \abs{u_r}^{\eta + 1} \textbf{1}_{\{|u_r|\leq K\}} \dd x \dd r =  0.
\end{align}
Notice that the left hand side of~\eqref{eq:EnergyEqualityApp} contains only nonnegative terms. Therefore each individual term needs to vanish. However,
\begin{align}
    \int_0^T \norm{\nabla u_r}_{L^p}^p \dd r = 0 \quad \Rightarrow \quad u \equiv 0,
\end{align}
which leads to the contradiction, since for $\eta + 1 \leq 0$, 
\begin{align}
   0 =  \int_0^T \int_{\Omega} \abs{u_r}^{\eta + 1} \textbf{1}_{\{|u_r|\leq K\}} \dd x \dd r \geq T \abs{ \Omega }.
\end{align}

\subsubsection*{Existence of robustified solution}
Let $H$ satisfy~\eqref{ex:ConditionH} and~$w^H$ be a fractional Brownian motion with Hurst parameter~$H$. We will check that the local time~$L$ associated to the fraction Brownian motion~$w^H$ is sufficiently regular for us to can apply Theorem~\ref{main theorem}.

We recall that a fractional Brownian motion with Hurst parameter~$H$ is a continuous centered Gaussian process whose covariance is given by 
\[
\mathbb{E}[w^H_s \otimes w^H_t]=\frac{1}{2}(t^{2H}+s^{2H}-|t-s|^{2H})\mbox{Id}.
\]
where $\mbox{Id}$ denotes the identity in $\R^N$. The space-time regularity of local times associated to locally non-deterministic Gaussian processes has been adressed by Harang and Perkowksi in~\cite{harang2020cinfinity} (see also~\cite{Harang2021} for regularity results for local times associated with Volterra-L\'evy processes). In particular \cite[Theorem 3.4]{harang2020cinfinity} shows that if
\begin{align}\label{app:ConditionLocal}
H \in (0,1/N) \quad \text{ and } \quad \lambda<\frac{1}{2H}-\frac{N}{2},
\end{align}
then almost any realization of the fractional Brownian motion~$w^H$ admits a local time $L$ that satisfies $L\in C^{0,\gamma}\big([0,T];H^{\lambda}(\mathbb{R}^N)\big)$ for all $\gamma\in [0, 1 -(\lambda + N/2)H)$. 

Notice that due to~\eqref{app:ConditionLocal}
\begin{align*}
    1 - (\lambda + N/2)H \geq 1/2.
\end{align*}
Therefore, if~\eqref{app:ConditionLocal} is satisfied we can always choose~$\gamma > 1/2$ for the temporal regularity. However, whether the spatial regularity is sufficient is to be determined.

If $\eta \in (-N/4,0)$ it holds~$b \in L^4(\mathbb{R}^N)$. Therefore $r = 2 = r'$ and we have as a condition
\begin{align}
    1 \leq \lambda < \frac{1}{2H} - \frac{N}{2} \quad \Rightarrow \quad H < \frac{1}{N+2}.
\end{align}

If $\eta \in (-N/2,-N/4]$ we need to leave the Hilbert scale and work with Sobolev spaces. Let $\varepsilon \ll 1$ and define~$r_\varepsilon := - \tfrac{N-\varepsilon}{2\eta}$. We observe that $b\in L^{2r_\varepsilon}(\mathbb{R}^N)$ for all $\varepsilon >0$. Moreover, $r_\varepsilon ' = \tfrac{N - \varepsilon}{N + 2\eta - \varepsilon} \rightarrow \tfrac{N}{N + 2\eta}$ as $\varepsilon \to 0$. Next, we transfer the Hilbert scale~$H^\lambda$ to the Sobolev scale~$W^{1,q}$ by an application of Sobolev embeddings. Let~$1 \leq \lambda < 1+N/2$ and
\begin{align*}
    1 - \frac{N}{q}\leq \lambda - \frac{N}{2} \quad \Leftrightarrow \quad q \leq \frac{2N}{2 + N - 2 \lambda}.
\end{align*}
Then $H^\lambda_{\mathrm{loc}}(\mathbb{R}^N)\hookrightarrow W^{1, q }_{\mathrm{loc}}(\mathbb{R}^N)$. In particular, if
\begin{align*}
  \frac{2N}{2 + N - 2 \lambda} > \frac{N}{N + 2\eta},
\end{align*}
then there exists~$\varepsilon >0$ such that $H^\lambda_{\mathrm{loc}}(\mathbb{R}^N)\hookrightarrow W^{1, r_\varepsilon' }_{\mathrm{loc}}(\mathbb{R}^N)$.
Therefore, we find as conditions
\begin{subequations} \label{app:ConditionsCoupled}
\begin{align}
    1 &\leq \lambda < 1 + \frac{N}{2}, \\
    1 &\leq \lambda < \frac{1}{2H} - \frac{N}{2}, \\ \label{app:HCoupled}
   1 - \frac{N + 4 \eta}{2} &< \lambda.
\end{align}
\end{subequations}

Since $1 - \frac{N + 4 \eta}{2}  \rightarrow 1 + N/2$ as $\eta \to -N/2$ monotonically, system~\eqref{app:ConditionsCoupled} is solvable if
\begin{align}
   1 - \frac{N + 4\eta}{2} <\frac{1}{2H} - \frac{N}{2} \quad \Rightarrow \quad H <\frac{1}{2- 4\eta}.
\end{align}

In summary, we have verified that fractional Brownian motions whose Hurst parameter satisfy~\eqref{ex:ConditionH} are sufficiently regularizing. The existence of a robustified solution follows by Theorem~\ref{main theorem}.

\subsection*{Local time and occupation times formula}
We recall for the reader the basic concepts of occupation measures, local times and the occupation times formula. A comprehensive review paper on these topics is \cite{horowitz}. 
\begin{definition}
Let $w:[0,T]\to \R^N$ be a measurable path. Then the occupation measure at time $t\in [0,T]$, written $\mu^w_t$ is the Borel measure on $\R^d$ defined by 
\[
\mu^w_t(A):=\lambda(\{ s\in [0,t]:\ w_s\in A\}), \quad A\in \mathcal{B}(\R^N),
\]
where $\lambda$ denotes the standard Lebesgue measure. 
\end{definition}
The occupation measure thus measures how much time the process $w$ spends in certain Borel sets. Provided for any $t\in [0,T]$, the measure is absolutely continuous with respect to the Lebesgue measure on $\R^N$, we call the corresponding Radon-Nikodym derivative local time of the process $w$:
\begin{definition}
Let $w:[0,T]\to \R^N$ be a measurable path. Assume that there exists a measurable function $L^w:[0,T]\times \R^N\to \R_+$ such that 
\[
\mu^w_t(A)=\int_A L^w_t(z)dz, 
\]
for any $A\in \mathcal{B}(\R^N)$ and  $t\in [0,T]$. Then we call $L^w$ local time of $w$. 
\end{definition}
Note that by the definition of the occupation  measure, we have for any bounded measurable function $f:\R^N\to \R$ that 
\begin{equation}
    \int_0^tf(w_s)ds=\int_{\R^N} f(z)\mu^w_t(dz).
    \label{occupation times formula}
\end{equation}
The above equation \eqref{occupation times formula} is called occupation times formula. Remark that in particular, provided $w$ admits a local time, we also have for any $u\in \R^N$
\begin{equation}
    \int_0^tf(u-w_s)ds=\int_{\R^N} f(u-z)\mu^w_t(dz)=\int_{\R^N}f(u-z)L^w_t(z)dz=(f*L^w_t)(u).
\end{equation}
\subsection*{Sewing Lemma}
Let us recall the Sewing Lemma due to \cite{gubi} (see also \cite[Lemma 4.2]{frizhairer}). Let $E$ be a Banach space, $[0,T]$ a given interval. Let $\Delta_n$ denote the $n$-th simplex of $[0,T]$, i.e. $\Delta_n:\{(t_1, \dots, t_n)| 0\leq t_1\dots\leq t_n\leq T \} $. For a function $A:\Delta_2\to E$ define the mapping $\delta A: \Delta_3\to E$ via
\[
(\delta A)_{s,u,t}:=A_{s,t}-A_{s,u}-A_{u,t}
\]
Provided $A_{t,t}=0$ we say that for $\alpha, \beta>0$ we have $A\in C^{\alpha, \beta}_2(E)$ if $\norm{A}_{\alpha, \beta}<\infty$ where
\begin{align*}
\norm{A}_\alpha &:=\sup_{(s,t)\in \Delta_2}\frac{\norm{A_{s,t}}_E}{|t-s|^\alpha}, \hspace{3em}
\norm{\delta A}_{\beta}:=\sup_{(s,u,t)\in \Delta_3}\frac{\norm{(\delta A)_{s,u,t}}_E}{|t-s|^\beta}, \\
\norm{A}_{\alpha, \beta}&:=\norm{A}_\alpha+\norm{\delta A}_\beta.
\end{align*}
In this case, we call $A$ a germ. For a function $f:[0,T]\to E$, we note $f_{s,t}:=f_t-f_s$. Moreover, if for any sequence $(\mathcal{P}^n([s,t]))_n$ of partitions of $[s,t]$ whose mesh size goes to zero, the quantity 
\begin{equation}
    (I^nA)_{s,t}=\sum_{[u,v]\in \mathcal{P}^n([s,t])}A_{u,v}
    \label{riemann sums}
\end{equation}
converges to the same limit, we note
\[
(\mathcal{I} A)_{s,t}:=\lim_{n\to \infty}\sum_{[u,v]\in \mathcal{P}^n([s,t])}A_{u,v}.
\]

\begin{lemma}[Sewing]
\label{sewing}
Let $0<\alpha\leq 1<\beta$. Then for any $A\in C^{\alpha, \beta}_2(E)$, $(\mathcal{I} A)$ is well defined and we say that the germ $A$ admits a sewing. Moreover, denoting $(\mathcal{I} A)_t:=(\mathcal{I} A)_{0,t}$, we have $(\mathcal{I} A)\in C^\alpha([0,T], E)$ and $(\mathcal{I} A)_0=0$ and for some constant $c>0$ depending only on $\beta$ we have
\[
\norm{(\mathcal{I} A)_{t}-(\mathcal{I} A)_{s}-A_{s,t}}_{E}\leq c\norm{\delta A}_\beta |t-s|^\beta.
\]
\end{lemma}
\begin{proof}
For readers unfamiliar with the Sewing Lemma, let us sketch the main argument why the Riemann-type sums \eqref{riemann sums} converge in the case of dyadic partitions. Let us thus consider $\mathcal{P}^n([0,t])=\{t_k=\frac{k}{2^n}t, \quad  k=1, \dots 2^n-1\}$ and show that the sequence $(I^nA)_t$ is Cauchy in $E$. Indeed, denoting by $u_k$ the midpoint of $t_k$ and $t_{k+1}$, remark that 
\[
(I^nA)_t-(I^{n+1}A)_t=\sum_k^{2^n} A_{t_k, t_{k+1}}-\sum_k^{2^n}A_{t_k, u_k}-\sum_k^{2^n}A_{u_k, t_{k+1}}=\sum_k^{2^n} (\delta A)_{t_k, u_k, t_{k+1}}
\]
   Using the triangle inequality and the assumption $A\in C_2^{\alpha, \beta}$, we have 
   \[
   \norm{(I^nA)_t-(I^{n+1}A)_t}_E\leq \norm{\delta A}_\beta 2^{n(1-\beta)}
   \]
   meaning that for $\beta>1$ we can obtain convergence of $(I^nA)_t$ in $E$.
\end{proof}

Let us finally cite a result allowing to commute limits and sewings. 
\begin{lemma}[Lemma A.2 \cite{Galeati2021}] 
\label{sewing convergence}
For $0<\alpha\leq 1<\beta$ and $E$ a Banach space, let  $A\in C^{\alpha, \beta}_2(E)$ and $ (A^n)_n\subset C^{\alpha, \beta}_2(E)$ such that for some $R>0$ $\sup_{n\in \mathbb{N}}\norm{\delta A^n}_\beta\leq R$ and such that $\norm{A^n-A}_\alpha\to 0$. Then \[\norm{\mathcal{I}(A-A^n)}_\alpha\to 0.\]
\end{lemma}

\subsection*{Some classical Lemmata from monotone operator theory}
Suppose $u^\epsilon$ is some approximation to \eqref{Original problem} (the solution to the Galerkin projected problem or the solution to the mollified problem \eqref{regularized problem II} for example) for which uniform bounds in $L^p(0, T; W^{1, p}_0(\Omega))$ hold, meaning we find $u$ such that $u^\epsilon\rightharpoonup u$ along a subsequence in $L^p(0, T; W^{1, p}_0(\Omega))$. Moreover, as $\mbox{div} S(\nabla u^\epsilon)$ will be bounded in $L^{p'}(0, T; W^{-1, p'}(\Omega))$, we find a $\xi$ such that $\mbox{div} S(\nabla u^\epsilon)\rightharpoonup \xi$ along a further subsequence in $L^{p'}(0, T; W^{-1, p'}(\Omega))$. For the identification step to work, we need an argument ensuring that $\xi=\mbox{div} S(\nabla u)$. This is precisely the content of Minty's Lemma. 
\begin{lemma}[Minty's Lemma/ Monotonicity trick {\cite[(25.4b)]{Zeidler1990}}]
    \label{minty}
    Let $X$ be a real reflexive Banach space and $A:X\to X^*$ a monotone and hemicontinuous operator. Then, provided we have 
    \begin{align*}
        u^n &\rightharpoonup u \quad in \ X\\
        Au^n &\rightharpoonup \xi \quad in \ X^*\\
        \langle Au^n, u^n\rangle_{X^*\times X}&\to \langle \xi, u\rangle_{X^*\times X} 
    \end{align*}
    we may conclude $Au=\xi$.
\end{lemma}

Concerning the identification of the other non-linear term, one proceeds differently. Indeed, for a generic non-linearity $b:\R^N\to \R^N$, we can only hope for $b(u^\epsilon)\rightharpoonup b(u)$ provided that the convergence $u^\epsilon\to u$ holds in the strong topology of some function space. Towards this end, we require some refined a priori bounds and an associated compact embedding, given by the following classical result. 

\begin{lemma}[An Aubin-Lions Lemma {\cite[Corollary~5]{MR916688}}] \label{lem:Aubin-Lions}
Given Banach spaces~$X$, $B$ and $Y$ and assume $X \hookrightarrow \hookrightarrow B \hookrightarrow Y$. Additionally, let $q,r \in [1,\infty]$. Then
\begin{enumerate}
    \item \label{item:Aubin-Lions-leb} if $q < \infty$ and $s > (1/r - 1/q) \vee 0$
        \begin{align} \label{eq:Aubin-Lions-leb}
            L^q(0,T;X) \cap W^{s,r}(0,T;Y) \hookrightarrow\hookrightarrow L^q(0,T;B),
        \end{align}
    \item \label{item:Aubin-Lions-cont} if $q = \infty$ and $s > 1/r$
        \begin{align} \label{eq:Aubin-Lions-cont}
            L^\infty(0,T;X) \cap W^{s,r}(0,T;Y) \hookrightarrow\hookrightarrow C^0(0,T;B).
        \end{align}    
\end{enumerate}
\end{lemma}

\subsection*{An algebraic inequality}
\begin{lemma}\label{app:Algebraic}
Let $a, C, K \in \mathbb{R}$ such that $K \geq -C^2/4$. Additionally, assume that
\begin{align}\label{app:AlgebraicAss}
    a^2 \leq K  + C a.
\end{align}
Then 
\begin{align} \label{app:Algebraic eq:Bound}
\frac{C}{2} - \sqrt{K + \frac{C^2}{4}} \leq a \leq \frac{C}{2} + \sqrt{K + \frac{C^2}{4}}.
\end{align}
\end{lemma}
\begin{proof}
Reordering~\eqref{app:AlgebraicAss} results in 
\begin{align*}
    a(a-C) \leq K.
\end{align*}
Completing the square ensures
\begin{align*}
    \left(a-\frac{C}{2} \right)^2 \leq K + \frac{C^2}{4}. 
\end{align*}
This implies
\begin{align*}
    \abs{a - \frac{C}{2}} \leq \sqrt{K + \frac{C^2}{4}}.
\end{align*}
The assertion is verified.
\end{proof}

If $K< C^2/4$, then the parabola $a \mapsto a(a-C)$ and the constant function $K$ do not intersect. Thus~\eqref{app:AlgebraicAss} wont be satisfied.

\section*{Acknowledgement}The first author acknowledges support from the European Research Council (ERC) under the European Union’s Horizon 2020 research
and innovation programme (grant agreement No. 949981). The second author acknowledges support from Deutsche Forschungsgemeinschaft (DFG, German Research Foundation) – SFB 1283/2 2021 – 317210226.

	\begin{figure}[ht]
	    \centering
	  \includegraphics[height=10mm]{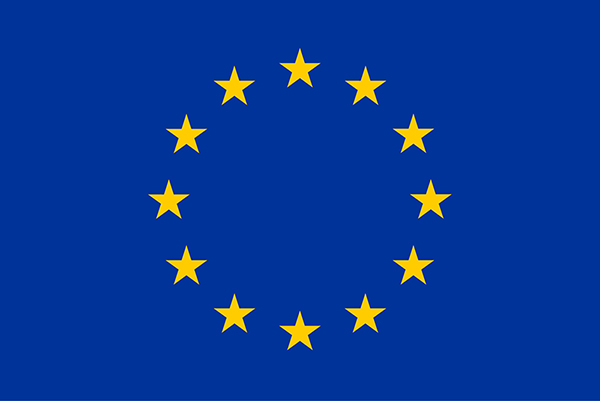}
	\end{figure}

\textbf{Declarations: }
All the authors declare that they have no conflicts of interest. Data availability statement is not applicable in the context to the present article since all the results here are theoretical in nature and do not involve any data.

\bibliographystyle{alpha}
\bibliography{main}

\end{document}